\documentclass{amsart}
\usepackage{amssymb,euscript,amsmath, mathrsfs, graphicx, subfigure}

\def\P{{\mathbb P}}
\def\CC{{\mathbb C}}
\def\FF{{\mathbb F}}
\def\RR{{\mathbb R}}
\def\ZZ{{\mathbb Z}}
\def\QQ{{\mathbb Q}}

\def\Spec{\operatorname{Spec}}

\def\fX{{\mathfrak X}}

\def\cC{{\mathcal C}}

\newcommand{\Div}{\operatorname{Div}\nolimits}

\DeclareMathOperator{\divisor}{div}
\DeclareMathOperator{\ord}{ord}

\DeclareMathOperator{\Pic}{Pic}

\DeclareMathOperator{\Sym}{Sym}

\numberwithin{equation}{section}
\newtheorem{thm}[equation]{Theorem}
\newtheorem{prop}[equation]{Proposition}

\newtheorem{cor}[equation]{Corollary}
\newtheorem{conj}[equation]{Conjecture}
\newtheorem*{BNT}{Brill-Noether Theorem}
\newtheorem*{TRR}{Tropical Riemann-Roch Theorem}
\newtheorem*{LT}{Luo's Theorem}
\newtheorem*{SL}{Specialization Lemma}
\theoremstyle{definition}
\newtheorem{defn}[equation]{Definition}

\newtheorem{ex}[equation]{Example}
\theoremstyle{remark}
\newtheorem{notn}[equation]{Notation}
\newtheorem{rem}[equation]{Remark}

\begin{document}

\title{A tropical proof of the Brill-Noether Theorem}
\author{Filip Cools}
\address{Department of Mathematics, K. U. Leuven, Celestijnenlaan 200B, B-3001 Leuven, Belgium \ \ \tt filip.cools@wis.kuleuven.be}
\author{Jan Draisma}
\address{Department of Mathematics and Computer Science, TU/E, PO Box 513, 5600 MB Eindhoven, The Netherlands  \ \ \tt j.draisma@tue.nl}
\author{Sam Payne}
\address{Mathematics Department, Yale University, 10 Hillhouse Ave, New Haven, CT, 06511 USA  \ \ \tt sam.payne@yale.edu}
\author{Elina Robeva}
\address{Department of Mathematics, Stanford University, Bldg 380, Stanford, CA 94305, USA  \ \ \tt erobeva@stanford.edu}

\begin{abstract}
We produce Brill-Noether general graphs in every genus, confirming a conjecture of Baker and giving a new proof of the Brill-Noether Theorem, due to Griffiths and Harris, over any algebraically closed field.
\end{abstract}

\maketitle

\section{Introduction}

Let $X$ be a smooth projective curve of genus $g$.  Brill-Noether theory studies the geometry of the subscheme $W^r_d(X)$ of $\Pic_d(X)$ parametrizing linear equivalence classes of divisors of degree $d$ that move in a linear system of dimension at least $r$, especially when the curve $X$ is general.  This subscheme can be realized as a degeneracy locus of a natural map of vector bundles, with naive expected dimension
\[
\rho = g - (r+1)(g-d+r).
\]
The Brill-Noether Theorem, due to Griffiths and Harris, says that the naive dimension count is essentially correct, for a general curve.

\begin{BNT}
[\cite{GriffithsHarris80}]  Suppose $X$ is general.
\begin{enumerate}
\item If $\rho$ is negative then $W^r_d(X)$ is empty.
\item If $\rho$ is nonnegative then $W^r_d(X)$ has dimension $\min\{\rho, g\}$.
\end{enumerate}
\end{BNT}

\noindent Curves that are easy to write down, such as complete intersections in projective spaces and Grassmannians, have many more special divisors than this dimension count predicts, so the generality hypothesis is crucial.  The fact that $W^r_d$ is nonempty and of dimension at least $\min \{ \rho, g \}$ for an arbitrary curve, when $\rho$ is nonnegative, is significantly easier, and was assumed by Griffiths and Harris in their proof of (2).  General results on degeneracy loci say that $W^r_d$ must support the expected cohomology class given by the Thom-Porteous determinantal formula in the Chern classes of the bundles.  An explicit computation, due to Kempf, Kleiman, and Laksov, shows that this expected class is a nonzero multiple of a power of the theta divisor \cite{Kempf71, KleimanLaksov72}.  

The nonexistence of special divisors when $\rho$ is negative, and the upper bound on the dimension of $W^r_d$ when $\rho$ is nonnegative, for a general curve, is considered much deeper.  The depth of this result is related to the difficulty of writing down a sufficiently general curve in high genus, or even a reasonable criterion for a curve to be sufficiently general.  The original proof uses a degeneration to a $g$-nodal rational curve, followed by a very subtle transversality argument for certain Schubert varieties associated to osculating flags of a rational normal curve.  Two subsequent proofs \cite{EisenbudHarris83c, Lazarsfeld86} also continue to be heavily cited after more than twenty years; the Eisenbud-Harris proof uses limits of linear series and a degeneration to a cuspidal rational curve.  Lazarsfeld's proof involves no degenerations; he shows that a general hyperplane section of a complex $K3$ surface of Picard number 1 is Brill-Noether general, by which we mean that $W^r_d$ is empty when $\rho$ is negative and of dimension $\min\{ \rho, g\}$ otherwise.  Here we give a novel ``tropical" proof of the Brill-Noether Theorem, replacing the subtle transversality arguments in the original proof with the combinatorics of chip-firing on certain graphs.  The graphs encode degenerations to semistable unions of rational curves without self-intersection, which may be realized as degenerations of the $g$-nodal rational curves used by Griffiths and Harris.

\smallskip

Our starting points are the theory of ranks of divisors on graphs, as developed by Baker and Norine in their groundbreaking paper \cite{BakerNorine07}, and Baker's Specialization Lemma \cite{Baker08}, which says that the dimension of the complete linear system of a divisor on a smooth curve over a discretely valued field is less than or equal to the rank of its specialization to the dual graph of the special fiber of a strongly semistable regular model.  This allows one to translate geometric results about existence of special divisors between curves and metric graphs.  For instance, when $\rho$ is nonnegative, the nonemptiness of $W^r_d$ for arbitrary $X$ implies that every metric graph $\Gamma$ of genus $g$ with rational edge lengths has a divisor of degree $d$ and rank at least $r$, and a rational approximation argument from \cite{GathmannKerber08} then shows that the same holds for metric graphs with arbitrary edge lengths.

For (1), the nonexistence part of the Brill-Noether Theorem, the natural implication goes in the other direction.  Suppose $X$ has a strongly semistable regular model whose special fiber has dual graph $\Gamma$.  If $\Gamma$ has no divisor of degree $d$ and rank at least $r$, then neither does $X$. Similarly, if there is an effective divisor of degree $r + \rho + 1$ on $\Gamma$ that is not contained in any effective divisor of degree $d$ and rank at least $r$, then there is such a divisor on $X$, and it follows that the dimension of $W^r_d(X)$ is at most $\rho$.  See Section~\ref{sec:classical} for further details.

The graph $\Gamma$ that we consider is combinatorially a chain of $g$ loops, as shown in Figure~\ref{fig:Gamma}, with generic edge lengths.  Here, generic means that the tuple of lengths $(\ell_1,\ldots,\ell_g,m_1,\ldots,m_g)$ in $\RR_{>0}^{2g}$ lies outside the union of a finite collection of hyperplanes.  See Definition~\ref{defn:Generic} for a precise statement.

\begin{figure}
\includegraphics{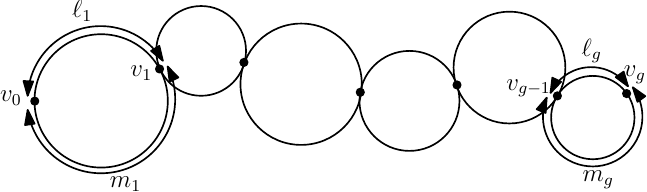}
\caption{The graph $\Gamma$ is a chain of $g$ loops with generic edge lengths.}
\label{fig:Gamma}
\end{figure}

\begin{thm} \label{thm:Main}
Suppose $\Gamma$ is a chain of $g$ loops with generic edge lengths.
\begin{enumerate}
\item If $\rho$ is negative then $\Gamma$ has no effective divisors of degree $d$ and rank at least $r$.
\item If $\rho$ is nonnegative then $\Gamma$ has no effective divisors of degree $d$ and rank at least $r$ that contain $(r + \rho + 1) \, v_0$.
\end{enumerate}
\end{thm}

\noindent The existence of graphs with no special divisors when $\rho$ is negative was conjectured by Baker.  In particular, Theorem~\ref{thm:Main}(1) confirms Conjectures~3.9(2), 3.10(2), and 3.15 of \cite{Baker08}.  As a consequence of the theorem, we obtain the following criterion for a curve to be Brill-Noether general.

\begin{cor}  \label{cor:criterion}
Let $X$ be a curve of genus $g$ over a discretely valued field with a regular, strongly semistable model whose special fiber has dual graph $\Gamma$. Then $X$ is Brill-Noether general.
\end{cor}

\noindent Such curves exist over any complete, discretely valued field \cite[Appendix~B]{Baker08}, and the existence of Brill-Noether general curves over an arbitrary algebraically closed field follows easily.  See Section~\ref{sec:classical}.

While our primary interest is the Brill-Noether Theorem, and the tropical criterion for a curve to be Brill-Noether general, we also prove tropical analogues of the dimension part of the Brill-Noether Theorem and an enumerative formula when $\rho$ is zero.  We write $W^r_d(\Gamma)$ for the subset of the real torus $\Pic_d(\Gamma)$ parametrizing divisor classes of degree $d$ and rank at least $r$.

\begin{thm}  \label{thm:dimension}
If $\rho$ is nonnegative, then the dimension of $W^r_d(\Gamma)$ is $\min\{\rho, g\}$.
\end{thm}

\noindent This points toward a potentially interesting Brill-Noether theory entirely within tropical geometry, and it is natural to wonder whether these results can be extended to a larger class of graphs.

When $\rho$ is zero and $X$ is general, Griffiths and Harris show that $W^r_d$ consists of finitely many reduced points, and the formula of Kempf, Kleiman, and Laksov says that the number of points is exactly
\[
\lambda = g! \prod_{i=0}^r \frac{i!}{(g-d+r+i)!}.
\]
This integer $\lambda$ has many interpretations.  It is the $(g-d+r)$th Catalan number of dimension $r$ \cite[p.~133]{MacMahon60}.  By the hook-length formula, it counts standard tableaux on the $(r+1) \times (g-d+r)$ rectangle \cite[Exercise~9, p.~54]{Fulton97}.  It is also the degree of the Grassmannian of $r$-planes in $\P^{g-d+2r}$ in its Pl\"ucker embedding, and hence counts $r$-planes in $\P^{g-d+2r}$ meeting $g$ general $g-d+r-1$ planes \cite[Lecture~19]{Harris92}.  We prove that it also counts divisor classes of degree $d$ and rank $r$ on $\Gamma$, with an explicit bijection to tableaux.

\begin{thm} \label{thm:Count}
If $\rho$ is zero, then there are exactly $\lambda$ distinct divisor classes of degree $d$ and rank $r$ on $\Gamma$.
\end{thm}

\noindent This exact equality is somewhat surprising; enumerative formulas in tropical geometry often require counts with multiplicities.  Here it seems that every divisor should be counted with multiplicity one.

\begin{conj} \label{conj:lifting}
Let $X$ be a smooth projective curve of genus $g$ over a discretely valued field, for which the special fiber of a strongly semistable regular model has dual graph $\Gamma$.  Then every divisor of degree $d$ and rank $r$ on $\Gamma$ lifts to a divisor of degree $d$ and rank $r$ on $X$.  Furthermore, if $\rho$ is zero, then this lift is unique.
\end{conj}

\begin{rem}
Important refinements of the Brill-Noether Theorem include Gieseker's proof of the Petri Theorem, which implies that $W^r_d(X)$ is smooth away from $W^{r+1}_d(X)$, for a general curve.  Fulton and Lazarsfeld then applied a general connectedness theorem to prove that $W^r_d(X)$ is irreducible \cite{FultonLazarsfeld81}.  It should be interesting to see if tropical methods may be applicable to these properties of $W^r_d$ as well.  Furthermore, there is a close analogy between Brill-Noether loci and degeneracy loci studied by Farkas, where the natural map from $\Sym^n(H^0(C, L))$ to $H^0(C, L^{\otimes n})$ drops rank, and the Strong Maximal Rank Conjecture predicts that these loci should have the naive expected dimension for a general curve \cite[Conjecture~5.4]{AproduFarkas11}.  Hyperplane sections of $K3$ surfaces are not sufficiently general for this conjecture, but it is tempting to hope that tropical methods may be useful instead.
\end{rem}

\begin{rem}
When the genus $g$ is small, the moduli space of curves is unirational over $\QQ$.  In these cases, rational points are dense in the moduli space and hence there exist Brill-Noether general curves defined over $\QQ$.  For large $g$, the moduli space is of general type, and Lang's Conjectures predict that rational points should be sparse.  In these cases, it is unclear whether there exists a Brill-Noether general curve defined over $\QQ$.  The rational numbers carry many discrete valuations, one for each prime, so the criterion given by Corollary~\ref{cor:criterion} could potentially be used to produce such curves by explicit computational methods; see Section~6 of \cite{BPR11} for details on computational tests for faithful tropical representations of minimal skeletons.  To the best of our knowledge, there are no known examples of Brill-Noether general curves defined over $\QQ$ when the moduli space is not unirational.
\end{rem}

\begin{rem}
We briefly mention a few related results that have appeared since this paper was written.  The special case of Conjecture~\ref{conj:lifting} where $r = 1$ and $\rho = 0$ has been proved by Coppens and the first author \cite[Theorem~2.3]{CoolsCoppens10}.  The two remaining open conjectures from \cite{Baker08}, on the existence of special divisors on complete graphs when $\rho$ is nonnegative, have now been proved by Caporaso \cite[Theorem~6.3]{Caporaso11b}.  The idea of using Theorem~\ref{thm:Main}(2) to prove the classical Brill-Noether Theorem has been developed into a theory of ``ranks" for tropical Brill-Noether loci, analogous to the Baker-Norine theory of ranks of divisors on graphs, by Lim, Potashnik, and the third author \cite{LPP11}.  They have also given examples of open subsets of the moduli space of metric graphs where the equality in Theorem~\ref{thm:dimension} does not hold.
\end{rem}

\noindent \textbf{Acknowledgments.}  This research was done in part during the special semester on Tropical Geometry at MSRI in Berkeley, and we are grateful for the hospitality and ideal working environment provided by this program.  We thank the participants of the chip-firing seminar at MSRI, including E.~Brugall\'e, E.~Cotterill, C.~Haase, E.~Katz, D.~Maclagan, G.~Musiker, J.~Yu, and I.~Zharkov, for many enlightening discussions, and M.~Baker, F.~Schroeter, and the referee for helpful comments on an earlier draft.

The work of FC is supported by a postdoctoral fellowship from the Research Foundation - Flanders. The work of JD was partially supported by MSRI and by a Vidi grant of the Netherlands Organisation for Scientific Research (NWO). The work of SP was partially supported by the Clay Mathematics Institute and NSF grant DMS 1068689. The work of ER was supported by a research grant from VPUE of Stanford University.

\section{Preliminaries}

We briefly recall the theory of divisor classes and ranks of divisors on (metric) graphs, following Baker and Norine \cite{BakerNorine07, Baker08}, to which we refer the reader for further details, references, and applications.

\subsection{Divisors, equivalence, and chip-firing}

Let $G$ be a finite, connected, undirected graph, with a positive
real number length assigned to each edge. For compatibility with
\cite{BakerNorine07}, we allow $G$ to have multiple edges but no loops.
Let $\Gamma$ be the associated metric graph, which is the compact
connected metric space obtained by identifying the edges of $G$ with
segments of the assigned lengths.  Such metric graphs are examples of
\emph{abstract tropical curves}, in the sense of \cite{GathmannKerber08}.
Let $g$ be the genus, or first Betti number, of $\Gamma$.

The group $\Div(\Gamma)$ is the free abelian group on the points of $\Gamma$, and elements of $\Div(\Gamma)$ are called \emph{divisors} on $\Gamma$.  The \emph{degree} of a divisor
\[
D = a_1 v_1 + \cdots + a_s v_s
\]
is the sum of the coefficients $\deg(D) = a_1 + \cdots + a_s$, and $D$ is \emph{effective} if each coefficient $a_i$ is nonnegative.

The subgroup of principal divisors are given by corner loci of piecewise linear functions, as follows.  Let $\psi$ be a continuous function on $\Gamma$, and suppose that there is a finite subdivision of $\Gamma$ such that $\psi$ is given by a linear function with integer slope on each edge of the subdivision.  Then, for each vertex $v$ of this subdivision, the order $\ord_v(\psi)$ is the sum of the incoming slopes of $\psi$ along the edges containing $v$, and the divisor of $\psi$ is
\[
\divisor(\psi) = \sum_v \ord_v(\psi) v.
\]
Two divisors $D$ and $D'$ are equivalent, and we write $D \sim D'$, if their difference $D - D'$ is equal to $\divisor(\psi)$ for some piecewise linear function $\psi$.

\begin{ex} \label{ex:Circle}
Let $\Gamma$ be a single loop, formed by two edges, of lengths $\ell$
and $m$ respectively, joining vertices $v$ and $v'$.  We label the points of $\Gamma$ by the interval
$[0,\ell + m)$ according their distance from $v$ in the counterclockwise
direction. In particular, $v$ is labeled $0$ and $v'$ is labeled
$m$. Let $D = kv + w$, where $w$ is either the zero divisor or the point of $\Gamma
\smallsetminus v$ labeled by $x \in (0, \ell + m)$, and assume $km$ is
not an integer multiple of $\ell + m$.  Then $D$ is linearly equivalent to
\[
D' =
\left \{
	\begin{array}{ll} (k-1)v' + w' & \mbox{ if $w$ is zero}, \\
	(k+1)v' & \mbox{ if $w$ is not zero and $x \equiv (k+1)m \mod \ell + m$}, \\
	kv' + w'' & \mbox{ otherwise}. \end{array}
\right.
\]
Here $w'$ and $w''$ are the points of $\Gamma \smallsetminus v'$ labeled by
$x' \equiv -(k-1)m$ and $x'' \equiv (x-km) \mod (\ell + m)$, respectively.
The equivalence can be seen easily from Dhar's burning algorithm with
base point $v'$, as presented in \cite{Luo11}.  It is also possible to
explicitly construct a piecewise linear function $\psi$ such that $D -
D' = \divisor (\psi)$, depending on the combinatorial configuration of
the points in question.  For instance, in the last case, if $x$ is greater than
$(k+1)m$, then $w''$ is at distance $km$ from $w$ on the segment
of length $\ell$, as shown.
\begin{center}
\includegraphics[scale=.6]{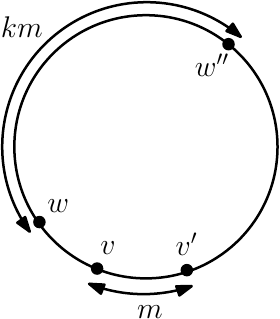} \hspace{2cm}  \includegraphics[scale=.6]{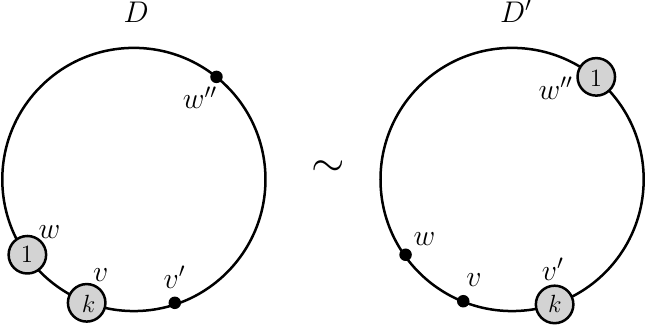}
\end{center}
A piecewise linear function $\psi$ such that $D - D' = \divisor(\psi)$ is then given by the constant functions $k$ and zero on
the segments $[v,w]$ and $[v',w'']$, respectively, and by linear functions
of slopes $-k$ and $-1$ on the segments $[v,v']$ and $[w,w'']$, respectively.
\end{ex}

The group of equivalence classes of divisors
\[
\Pic(\Gamma) = \Div(\Gamma) /  \sim
\]
is an extension of a real torus of dimension $g$ by $\ZZ$ \cite{MikhalkinZharkov08, BakerFaber11}.  The projection to $\ZZ$ takes the class of a divisor $D$ to its degree, which is well-defined because the degree of the divisor of any piecewise linear function is zero.  We write $\Pic_d(\Gamma)$ for the space of divisor classes of degree $d$ on $\Gamma$.

\begin{rem}
In combinatorics it is customary to refer to divisors on $\Gamma$ (especially those supported on the vertices of $G$) as \emph{chip configurations}.  The equivalence relation $\sim$ is generated by certain elementary equivalences called \emph{chip-firing moves}.  One imagines that $D = a_1 v_1 + \cdots + a_s v_s$ is represented by a stack of $a_i$ chips at each $v_i$ (or $|a_i|$ antichips, if $a_i$ is negative).  If $D$ is equivalent to $D'$ then $D - D' = \divisor (\psi)$ for some piecewise linear function $\psi$, and one imagines a path $\psi_t$ from zero to $\psi$ in the space of piecewise linear functions on $\Gamma$.  Then
\[
D_t = D - \divisor(\psi_t)
\]
is a path from $D$ to $D'$ in the space of divisors equivalent to $D$.  The laws of the game allow chips and antichips to collide and annihilate each other, and pairs of chips and antichips may sometimes be created from the ether.  However, if $D$ and $D'$ are both effective, Dhar's algorithm \cite{Dhar90} ensures that $\{ \psi_t \}$ can be chosen so that $D_t$ is effective for all $t$.  In other words, two effective divisors are linearly equivalent if and only if one can continuously move the chips from one configuration to the other by a sequence of allowable chip-firing moves.  The survey article \cite{HLMPPW08} is an excellent entry point to the vast literature on the combinatorics of chip-firing on graphs, and \cite{GathmannKerber08} and \cite{HladkyKralNorine08} include helpful explanations on the generalization of chip-firing to tropical curves and metric graphs.
\end{rem}

\subsection{Ranks of divisors}

Let $D$ be a divisor on $\Gamma$.  If $D$ is not equivalent to an effective divisor then the \emph{rank} of $D$, written $r(D)$, is defined to be $-1$.  Otherwise, $r(D)$ is the largest nonnegative integer $r$ such $D - E$ is equivalent to an effective divisor for every effective divisor $E$ of degree $r$ on $\Gamma$.

\begin{rem}
This notion of rank is a natural analogue of the dimension of the complete
linear system of a divisor on an algebraic curve.  A divisor $D$ on a
smooth projective curve $X$ moves in a linear series of dimension at
least $r$ if and only if $D - E$ is linearly equivalent to an effective
divisor for every effective divisor $E$ of degree $r$ on $X$.
\end{rem}

\begin{rem}
The set of effective divisors on $\Gamma$ that are equivalent to $D$ is naturally identified with the underlying set of a finite polyhedral complex \cite{HMY09}.  The dimension of this complex is bounded below by the rank of $D$, but is often larger.
\end{rem}

The canonical divisor $K$ on $\Gamma$ is defined as
\[
K = \sum_v (\deg v - 2) v,
\]
where the sum is over the vertices of $\Gamma$.  The degree of $K$ is $2g-2$, as can be checked by a computation of the topological Euler characteristic of $\Gamma$, and the rank of $K$ is $g -1$.  The latter is a special case of the following generalization to metric graphs of the Baker-Norine-Riemann-Roch Theorem for graphs.

\begin{TRR}[\cite{GathmannKerber08, MikhalkinZharkov08}]
Let $D$ be a divisor on $\Gamma$.  Then
\[
r(D) - r(K - D) = \deg(D) + 1 - g.
\]
\end{TRR}

\noindent This formula has many beautiful and useful applications.  For instance, any divisor $D$ of degree greater than $2g-2$ has rank exactly $\deg(D) -g$.

\subsection{Reduced divisors and Luo's Theorem}

Two fundamental tools for computing ranks of divisors on graphs are the existence and uniqueness of $v$-reduced divisors and Luo's Theorem on rank determining sets, which we now recall.

If we fix a basepoint $v$ on $\Gamma$, then each divisor $D$ on $\Gamma$ is equivalent to a unique \emph{$v$-reduced divisor}, denoted $D_0$ \cite[Theorem~10]{HladkyKralNorine08}.  This $v$-reduced divisor $D_0$ is characterized by two properties.  First, it is effective away from $v$, so $D_0 + k v$ is effective for $k$ sufficiently large.  Second, the points in $D_0$ are, roughly speaking, as close to $v$ as possible.  More precisely, the multiset of distances to $v$ of points in $D_0 + kv$ is lexicographically minimal among the multisets of distances to $v$ for all effective divisors equivalent to $D + kv$.  In particular, the coefficient of $v$ in $D_0$ is maximal among all divisors equivalent to $D$ and effective away from $v$, and hence $D$ is linearly equivalent to an effective divisor if and only if $D_0$ is effective.

\begin{rem}
The existence and uniqueness of $v$-reduced divisors on graphs is a natural analogue of the following existence and uniqueness property for divisors on algebraic curves.  If $D$ is a divisor on a smooth projective curve $X$ with a chosen basepoint $x$, then there is a unique divisor $D_0$ that is linearly equivalent to $D$ and effective away from $x$, and whose coefficient of $x$ is maximal among all such divisors.  If $D$ is effective, then $D_0$ is the zero locus of the section of $\mathcal O_X(D)$ with maximal order vanishing at $x$, which is unique up to scaling.
\end{rem}

\begin{ex} \label{ex:Reduced}
Let $\Gamma$ be the chain of loops shown in Figure~\ref{fig:Gamma},
so $v_n$ is the point of intersection between the $n$th loop and
the $(n+1)$th loop, for $1 \leq n \leq g-1$, and $v_0$ and $v_g$
are distinguished points on the first and last loops, respectively.
We characterize $v_n$-reduced divisors on $\Gamma$, as follows.
Let $\gamma'_i$ be the $i$th loop minus $v_{i}$, and let $\gamma_j$
be the $j$th loop minus $v_{j-1}$.  Then $\Gamma$ can be decomposed as
a disjoint union
\[
\Gamma  = \bigg( \bigsqcup_{i=1}^n \gamma'_i \bigg) \sqcup \ v_n \ \sqcup \bigg( \bigsqcup_{j=n+1}^g \gamma_j \bigg),
\]
as shown.

\bigskip
\begin{center}
\includegraphics{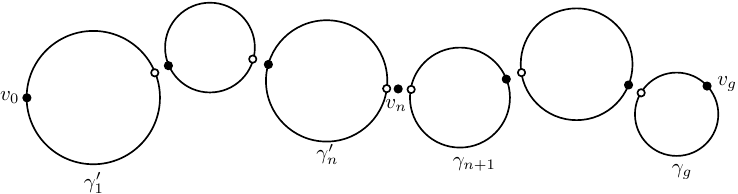}
\end{center}
\bigskip

\noindent A divisor $D$ on $\Gamma$ is $v_n$-reduced if and only if it is effective away from $v_n$ and each cell $\gamma'_i$ for $1 \leq i \leq n$, and $\gamma_j$ for $n+1 \leq j \leq g$, contains at most one point of $D$.  Indeed, if one of these cells contains more than one point of $D$ then they can be moved closer to $v_n$ by an equivalence similar to the one given in Example~\ref{ex:Circle} (with $\psi$ extended by a locally constant function on the complement of that cell), or by applying Dhar's burning algorithm. Conversely, if each of these cells contains at most one point of $D$, then the fact that $D$ is $v_n$-reduced can be checked by Dhar's burning algorithm.  See Section~2 of \cite{Luo11}.
\end{ex}

The existence and uniqueness of $v$-reduced divisors facilitate checking whether any given divisor is equivalent to an effective divisor.  However, to check if a divisor $D$ has rank at least $r$, in principle we must check whether $D - E$ is equivalent to an effective divisor for all divisors $E$ of rank $r$, of which there are uncountably many if $r$ is positive.  Luo has recently improved this situation by showing that there is a small, finite set of points in $\Gamma$ with the property that, for any $D$ and any $r$, it is enough to check for divisors $E$ whose support is contained in $A$.  Here, the \emph{support} of an effective divisor is the set of points that appear in it with nonzero coefficient.

\begin{LT}[\cite{Luo11}]
Let $A$ be a finite subset of $\Gamma$ such that the closure in $\Gamma$ of each connected component of $\Gamma \smallsetminus A$ is contractible.  Let $D$ be a divisor on $\Gamma$ and suppose that, for every effective divisor $E$ of degree $r$ whose support is contained in $A$, the difference $D - E$ is effective.  Then the rank of $D$ is at least $r$.
\end{LT}

\begin{rem}
Such a subset $A$ can always be chosen with size at most $g+1$.  Luo's Theorem is then a natural analogue of the following fact about divisors on curves.  Let $A = \{x_0 , \ldots, x_g \}$ be a set of distinct points on a smooth projective curve of genus $g$.  Then a divisor $D$ moves in a linear series of dimension at least $r$ if and only if $D - E$ is linearly equivalent to an effective divisor for every effective divisor $E$ with support in $A$.  For a proof, due to Varley, see Remark~3.13 of \cite{Luo11}.
\end{rem}

\begin{ex}
Let $\Gamma$ be the chain of loops shown in Figure~\ref{fig:Gamma}, and let $A = \{v_0, \ldots, v_g\}$.  Then the closure of each connected component of $\Gamma \smallsetminus A$ is either the top half or the bottom half of one of the loops, and hence is contractible.  In this case, Luo's Theorem says that a divisor $D$ on $\Gamma$ has rank at least $r$ if and only if $D - E$ is equivalent to an effective divisor for any effective divisor $E = r_0 v_0 + \cdots + r_g v_g$ of degree $r$.
\end{ex}

\subsection{Specialization}

We conclude this preliminary section with a review of Baker's Specialization Lemma, which relates dimensions of complete linear series on certain curves over discretely valued fields to ranks of divisors on graphs.

Let $K$ be a discretely valued field, with valuation ring $R$ and residue field $k$, and let $X$ be a smooth projective curve over $K$.  A \emph{strongly semistable regular model} of $X$ is a regular scheme $\fX$ over $\Spec R$ whose general fiber $\fX_K$ is isomorphic to $X$ and whose special fiber $\fX_k$ is a reduced union of geometrically irreducible smooth curves $X_0, \ldots, X_s$ that meet only at simple nodes defined over $k$.  The dual graph $G$ of the special fiber has vertices $v_0, \ldots, v_s$ corresponding to the irreducible components of $\fX_k$ and one edge joining $v_i$ to $v_j$ for each point of intersection in $X_i \cap X_j$.  Let $\Gamma$ be the associated metric graph, where each edge is assigned length 1.

Each point in $X(K)$ specializes to a smooth point in the special fiber.  We write $\tau: X(K) \rightarrow \Gamma$ for the induced map which takes a point $x$ to the vertex $v_i$ corresponding to the irreducible component of $\fX_k$ that contains the specialization of $x$.  This map $\tau$ is compatible with finite field extensions, as follows.  If $K'$ is a finite extension of $K$ then there is a unique relatively minimal strongly semistable regular model $\fX'$ of $X \times_K K'$ that dominates $\fX \times_R R'$.  Let $G'$ be the dual graph of the special fiber of $\fX'$, and let $\Gamma'$ be the associated metric graph in which each edge is assigned length $1/e$, where $e$ is the ramification index of $K'/K$.  Then $\Gamma'$ is naturally isomorphic to $\Gamma$, and the induced specialization maps $X(K') \rightarrow \Gamma$ for all finite extensions $K'/K$ together give a well-defined geometric specialization map
\[
\tau: X(\overline K) \rightarrow \Gamma.
\]
Furthermore, the induced map on free abelian groups $\tau_* : \Div (X_{\overline K}) \rightarrow \Div(\Gamma)$ respects linear equivalence, and hence descends to a degree preserving group homomorphism
\[
\tau_* : \Pic(X_{\overline K}) \rightarrow \Pic(\Gamma).
\]
See Section~2 of \cite{Baker08} for details, further references, and a proof of the following.

\begin{SL}
Let $D$ be a divisor on $X_{\overline K}$.  Then
\[
r(\tau_*(D)) \geq r(D),
\]
where $r(D)$ is the dimension of the complete linear system $|D|$ on $X_{\overline K}$.
\end{SL}

\noindent In particular, if $\Gamma$ has no divisors of degree $d$ and rank at least $r$, then the Brill-Noether locus $W^r_d$ in $\Pic_d(X)$ is empty.

\section{From tropical to classical Brill-Noether theory}  \label{sec:classical}

Here we explain how the classical Brill-Noether theorem, over $\CC$ or an arbitrary algebraically closed field, follows from Theorem~\ref{thm:Main}.  First, we deduce Corollary~1.2, which establishes the existence of Brill-Noether general curves over an arbitrary  complete, discretely valued field.

\begin{proof}[Proof of Corollary~\ref{cor:criterion}]  Suppose $X$ is a curve of genus $g$ over a discretely valued field $K$ with a regular semistable model whose special fiber has dual graph $\Gamma$.  By Theorem~\ref{thm:Main}(1), if $\rho$ is negative then $\Gamma$ has no divisors of degree $d$ and rank at least $r$.  Therefore, by Baker's Specialization Lemma, the Brill-Noether locus $W^r_d(X)$ is empty.  It remains to show that if $\rho$ is nonnegative then $\dim W^r_d(X)$ is at most $\rho$.

If $X$ has a divisor $D$ of degree $d$ with $r(D) \geq r$, then every effective divisor of degree $r$ on $X$ is contained in an effective divisor of degree $d$ and rank $r$ that is equivalent to $D$.  If the class of $D$ moves in a $1$-dimensional algebraic family in $W^r_d(X)$, then every effective divisor of degree $r+1$ on $X$ is contained in an effective divisor whose class lies in that family, and a straightforward induction on dimension shows that every effective divisor of degree $r + \dim W^r_d(X)$ is contained in an effective divisor whose class lies in $W^r_d(X)$.  Therefore, to prove that $\dim W^r_d(X)$ is at most $\rho$, it will suffice to produce an effective divisor of degree $r + \rho + 1$ that is not contained in any effective divisor of degree $d$ such that $r(D) \geq r$.

Now, choose a point $x$ in $X_{\overline K}$ specializing to $v_0$.  Then $(r + \rho +1)x$ specializes to $(r + \rho + 1)v_0$, which is not contained in any effective divisor of degree $d$ and rank at least $r$, by Theorem~\ref{thm:Main}(2).  By Baker's Specialization Lemma, it follows that $(r + \rho + 1)x$ is not contained in any effective divisor $D$ such that $r(D) \geq r$.
\end{proof}

\begin{rem}
The above proof uses the fact that $v_0$ is a vertex of $\Gamma$, and hence is in the image of $X_{\overline K}$.  These arguments can be extended to show that if $X'$ is a curve over a discretely valued field with a regular semistable model whose special fiber has dual graph $\Gamma'$ then every effective divisor of degree $r + \dim W^r_d(X')$ on $\Gamma'$  is contained in an effective divisor of degree $d$ and rank at least $r$.  See \cite{LPP11} for details.
\end{rem}

\begin{proof}[Brill-Noether Theorem over $\CC$]
The existence of Brill-Noether general curves over $\CC$ follows easily from Corollary~\ref{cor:criterion}.  The field of Laurent series $\QQ((t))$ is complete and discretely valued.  Therefore, Corollary~\ref{cor:criterion} and the existence theorem from Appendix B of \cite{Baker08} show that for each genus $g$ there is a Brill-Noether general curve $X$ of genus $g$ over $\QQ((t))$.  The algebraic closure of $\QQ((t))$ is isomorphic to $\CC$ as an abstract field, because any two uncountable algebraically closed fields of the same cardinality and characteristic are isomorphic.  Therefore, there is a (noncontinuous) embedding of $\QQ((t))$ as a subfield of $\CC$.  Fix such an embedding.  Then
\[
X_\CC = X \times_{\Spec \QQ((t))} \Spec \CC
\]
is a Brill-Noether general curve over $\CC$.
\end{proof}

The argument above is insufficient to prove the Brill-Noether Theorem over fields such as $\overline \FF_p$ that have no subfields with nontrivial valuations.  Nevertheless, the existence of Brill-Noether general curves over $\overline \FF_p$, or an arbitrary algebraically closed field, follows from Corollary~\ref{cor:criterion} and the existence of the moduli space of curves.

\begin{proof}[Brill-Noether Theorem over an algebraically closed field]
The coarse moduli space $M_g$ of smooth projective curves of genus $g$ is a scheme of finite type over $\Spec \ZZ$, and the locus $U$ of Brill-Noether general curves is Zariski open and hence is also a scheme of finite type over $\Spec \ZZ$.  By Corollary~\ref{cor:criterion}, the scheme $U$ has points over $\QQ((t))$ and $\FF_p((t))$ for all primes $p$, and hence surjects onto $\Spec \ZZ$.  Therefore, the fiber of $U$ over any prime field is a nonempty scheme of finite type, and hence $U$ has points over any algebraically closed field. 
\end{proof}

\section{Chip-firing on a generic chain of loops}

Let $\Gamma$ be a chain of $g$ loops, as pictured in Figure~\ref{fig:Gamma}.  The top and bottom segments of the $i$th loop connect the vertex $v_{i-1}$ to $v_i$ and have length $\ell_i$ and $m_i$, respectively.  In particular, the total length of the $i$th loop is $\ell_i + m_i$.

\begin{defn} \label{defn:Generic}
The graph $\Gamma$ is \emph{generic} if none of the ratios $\ell_i / m_i$ is equal to the ratio of two positive integers whose sum is less than or equal to $2g-2$.
\end{defn}

\noindent Some genericity condition on these lengths is necessary for nonexistence of special divisors.  For instance, if $\ell_i = m_i$ for all $i$ then $\Gamma$ is hyperelliptic \cite{BakerNorine09}, meaning that it has a divisor of degree two and rank one.  Our genericity condition is easily achieved with integer edge lengths.  For instance, one may take $\ell_i = 2g-2$ and $m_i = 1$ for all $i$.

\begin{notn}
Throughout the remainder of the paper, we assume that $\Gamma$ is generic, in the sense of Definition~\ref{defn:Generic}.  Our main results, stated in the introduction, are trivial when $g$ is zero or one and follow from the Tropical Riemann-Roch Theorem when $d$ is greater than $2g-2$, so we also assume that $g$ is at least two and $d$ is at most $2g-2$.
\end{notn}

Our main combinatorial tool in the proof of Theorems~\ref{thm:Main}, \ref{thm:dimension}, and \ref{thm:Count} is a  \emph{lingering lattice path} associated to each $v_0$-reduced divisor of degree $d$ on $\Gamma$, defined as follows.

\begin{defn}
A lingering lattice path $P$ in $\ZZ^r$ is a sequence $p_0, \ldots, p_g$ of points in $\ZZ^r$ such that each successive difference $p_i - p_{i-1}$ is either $(-1,\ldots,-1)$, a standard basis vector, or zero.
\end{defn}

\noindent When $p_i -p_{i-1}$ is zero, we say that the lattice path \emph{lingers} at the $i$th step.

Recall that a divisor $D$ on $\Gamma$
is $v_0$-reduced if and only if it is effective away from $v_0$ and each cell $\gamma_i$ in the decomposition given in Example~\ref{ex:Reduced}
contains at most one point of $D$.  We label such a point by its distance
from $v_{i-1}$ in the counterclockwise direction, so $v_i$ is labeled
by $m_i$.  This leads to a natural bijection
\[
\{ v_0\mbox{-reduced divisors on } \Gamma \} \longleftrightarrow \ZZ \times \RR/(\ell_1 + m_1) \times \cdots \times \RR/(\ell_g + m_g),
\]
taking a $v_0$-reduced divisor $D$ to the data $(d_0; x_1, \ldots, x_g)$, where $d_0$ is the coefficient of $v_0$ in $D$, and $x_i$ is the location of the unique point in $D$ on $\gamma_i$, if there is one, and zero otherwise.  The lingering lattice path associated to $D$ is defined in terms of this data as follows.  We label the coordinates of $\ZZ^r$ and $\RR^r$ from zero to $r-1$, and write $p_i(j)$ for the $j$th coordinate of $p_i$.  We write $\cC$ for the open Weyl chamber
\[
\cC = \{y \in \RR^r \ | \ y(0) > \cdots > y(r-1) > 0\}.
\]

\begin{defn}
Let $D$ be the $v_0$-reduced divisor of degree $d$ corresponding to $(d_0; x_1, \ldots, x_g)$.  Then the associated lingering lattice path $P$ in $\ZZ^r$ starts at $(d_0, d_0 - 1, \ldots, d_0 - r +1)$ with steps given by
\[
p_{i} - p_{i-1} =
	\left \{ \begin{array}{ll} (-1, \ldots, -1) & \mbox{if } x_i = 0; \\
	e_j & \mbox{if } x_i \equiv (p_{i-1}(j)+1)m_i \mod \ell_i + m_i, \mbox{ and}\\
	& \ \ \ p_{i-1} \mbox{ and } p_{i-1} + e_j \mbox{ are in } \cC; \\
	0 & \mbox{otherwise,} \end{array} \right.
\]
where $e_0,\ldots,e_{r-1}$ are the standard basis vectors in $\ZZ^r$.
\end{defn}

\noindent To see that the $i$th step in the lingering lattice path is well-defined when $p_{i-1}$ is in $\cC$, one uses the genericity condition on $\Gamma$, as follows.  Note first that the coordinates of each $p_i$ satisfy the inequalities $p_i(0) > \cdots > p_i(r-1)$.  This is because $p_0$ satisfies these inequalities, and the inequalities are preserved by each step in the lingering lattice path.  Furthermore, since only $d - d_0$ of the $x_i$ are nonzero and the coordinates of $d_0$ are bounded above by $d_0$, all of the coordinates of each $p_i$ are bounded above by $d$, and hence by $2g-2$.  In particular, if $p_{i-1}$ is in $\cC$ then its coordinates are distinct integers between 1 and $2g-2$. The genericity condition then ensures that $x_i \equiv (p_{i-1}(j)+1)m_i \mod \ell_i+m_i$ holds for at most one $j$, and hence the $i$th step is well-defined, as required.

As noted above, the coordinates of each point in the lingering lattice path are strictly decreasing, so $p_i$ is in the chamber $\cC$ if and only if its last coordinate is positive.  Note also that if $p_{i-1}$ is in $\cC$ then $p_{i-1} + e_j$ is not in $\cC$ exactly when $p_{i-1}(j-1)$ is only one more than $p_{i-1}(j)$.  Since $p_0(j-1)$ is only one more than $p_0(j)$, the condition on $p_{i-1} + e_j$  guarantees that, among the first $i$ steps of $P$, the number of steps in direction $e_j$ is always less than or equal to the number in direction $e_{j-1}$, for all $i, j \geq 1$.

\begin{rem}
The three cases for $p_i - p_{i-1}$ are parallel to the three cases for $D'$ in Example~\ref{ex:Circle} and the three cases in the proof of Theorem~\ref{thm:Technical}, below.  For an interpretation in terms of a chip-firing game on $\Gamma$, see Remark~\ref{rem:Game}.
\end{rem}

Let $D$ be a $v_0$-reduced divisor of degree $d$ on $\Gamma$, and let $P$
be the associated lingering lattice path in $\ZZ^r$.  Our main technical
result is then the following.

\begin{thm} \label{thm:Technical}
The divisor $D$ has rank at least $r$ if and only if the associated lingering lattice
path $P$ lies entirely in the open Weyl chamber $\cC$.
\end{thm}

\noindent Before giving the proof of Theorem~\ref{thm:Technical}, we explain how it implies Theorems~\ref{thm:Main}, \ref{thm:dimension}, and \ref{thm:Count}, and give examples illustrating the bijection between divisors and tableaux when $\rho$ is zero.

\begin{proof}[Proof of Theorem~\ref{thm:Main}]
Suppose $D$ is a divisor of degree $d$ and rank at least $r$ on $\Gamma$, and let $(d_0; x_1, \ldots, x_g)$ be the data associated to the $v_0$-reduced divisor equivalent to $D$.  We must show that $\rho = g - (r+1)(g-d+r)$ is nonnegative and $d_0$ is less than or equal to $r + \rho$.

Exactly $d-d_0$ of the $x_i$ are nonzero, so the lingering lattice path $P$ includes $g-d + d_0$ steps in the direction $(-1, \ldots, -1)$.  In particular, the last coordinate of $p_g$ is
\[
p_g(r-1) = d - g- r + 1 + \#\{ \mbox{steps of $P$ in direction $e_{r-1}$} \}.
\]
\noindent By Theorem~\ref{thm:Technical}, the lattice path $P$ lies
in the open Weyl chamber $\cC$, so $p_g(r-1)$ is strictly positive.
Therefore $P$ includes at least $g-d+r$ steps in direction $e_{r-1}$.
By construction, the number of steps of $P$ in the $e_i$ direction is
at least the number of steps in the $e_{i+1}$ direction, for all $i$,
so $P$ must include at least $g-d+r$ steps in each of the $r$ coordinate
directions.  Therefore, the total number steps is
\[
g \geq r(g-d+r) + (g-d +d_0). 
\]
Rearranging terms then shows that $d_0\leq r+\rho$, which proves part (2) of the theorem. It remains to show that $\rho$ is nonnegative.

Since $D$ is $v_0$-reduced and of rank at least $r$, $D - rv_0$ must be effective.  Hence $d_0$ is at least $r$, and $\rho\geq 0$, as required.
\end{proof}

\begin{proof}[Proof of Theorem~\ref{thm:dimension}]
Each divisor class of degree $d$ and rank $r$ on $\Gamma$ is associated to some lattice path in $\cC$, and the proof of Theorem~\ref{thm:Main} above shows that this path must include at least $g-d+r$ steps in each of the $r$ coordinate directions, as well as at least $g-d+r$ steps in the direction $(-1,\ldots, -1)$.  Therefore, the lattice path has at most $\rho = g-(r+1)(g-d+r)$ lingering steps.  Given a lingering lattice path $p_0, \ldots, p_g$ with $\min\{ \rho, g \}$ steps, the $v_0$-reduced realizations of this lattice path are chip configurations with $p_0(0)$ chips at $v_0$, a chip in location $x_i \equiv (p_{i-1}(j) + 1)m_i \mod \ell_i + m_i$ in the cell $\gamma_i$ if $p_i - p_{i-1} = e_j$, a chip anywhere else in $\gamma_i$ if the $i$th step is lingering, and no chips on $\gamma_i$ otherwise.  Since there are finitely many such paths, and the dimension of the space of realizations is equal to the number of lingering steps, it follows that the dimension of $W^r_d(\Gamma)$ is $\min \{ \rho, g \}$, as required.
\end{proof}

\begin{proof}[Proof of Theorem~\ref{thm:Count}]
Suppose $\rho$ is zero.  Then each $v_0$-reduced divisor $D$ of degree $d$
and rank $r$ has $d_0 = r$ and the associated lattice path $P$ has exactly
$g-d+r$ steps in each of the coordinate directions, $g-d+r$ steps in the
direction $(-1, \ldots, -1)$, and no lingering steps.  Moreover, each such
lattice path corresponds to a unique divisor of degree $d$ and rank $r$.
Therefore, there is a natural bijection between divisors of degree $d$
and rank $r$ on $\Gamma$ and $g$-step lattice paths from $(r, \ldots,
1)$ to itself in the open Weyl chamber $\cC$.  These lattice paths are in
natural bijection with standard tableaux on the rectangular shape $(r+1)
\times (g-d+r)$, as follows.  We label the columns of the tableau from
zero to $r$.  Then the number $i$ appears in the column $j$ column of
the tableau corresponding to $P$ for $0 \leq j \leq r-1$ if the $i$th
step of $P$ is in the $j$th coordinate direction, and in column $r$
if the $i$-th step of $P$ is in the direction $(-1,\ldots,-1)$.
\end{proof}

\begin{ex}
Consider the case where $(g,r,d) = (4,1,3)$.  Then $\rho$ is zero and $\lambda$ is two, corresponding to the classical fact that there are exactly two lines meeting four general lines in $\P^3$.  Theorem~\ref{thm:Count} says that there are exactly two $v_0$-reduced divisors of degree three and rank one on $\Gamma$, corresponding to the lattice paths $1,2,3,2,1$ and $1,2,1,2,1$ in $\ZZ$, respectively.

The first path is associated to the divisor
\begin{center}
\includegraphics[scale=.6]{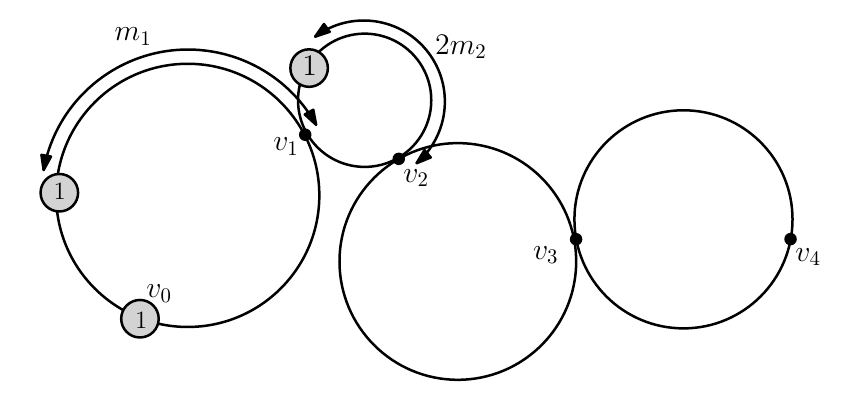}
\end{center}
which has points on the first two loops, since the first two steps of the lattice path are in the positive direction.  This corresponds to the tableau
\[
\begin{tabular}{|c|c|}
\hline
1&3\\
\hline
2&4\\
\hline
\end{tabular}
\]
in which one and two appear in the first column.  The second path is associated to the divisor
\begin{center}
\includegraphics[scale=.6]{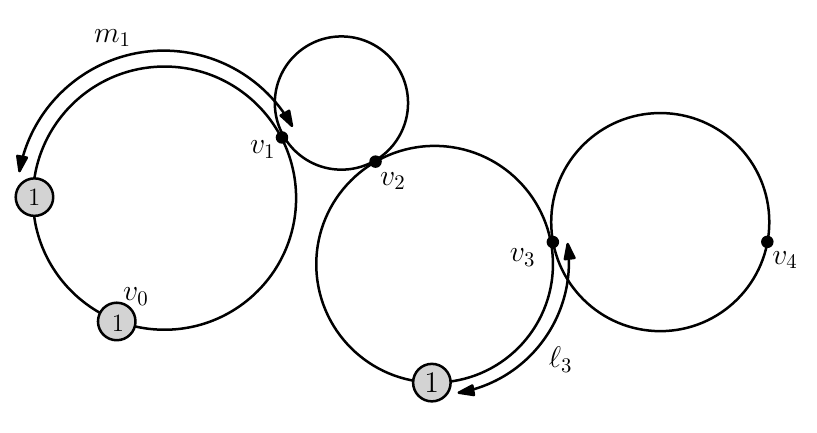}
\end{center}
which has points on the first and third loops.  This corresponds to the tableau
\[
\begin{tabular}{|c|c|}
\hline
1&2\\
\hline
3&4\\
\hline
\end{tabular}
\]
in which one and three appear in the first column.  It is a pleasant exercise to use the chip-firing moves in Example~\ref{ex:Circle} to show that each of these divisors has rank one.
\end{ex}

\begin{ex} \label{ex:g12r3d12}
Consider the case $(g,r,d) = (12,3,12)$.  Then $\rho$ is zero and $\lambda$ is 462.  The $v_0$-reduced divisors of degree 12 and rank 3 correspond to the 462 standard tableaux on a $4 \times 3$ rectangle.  For instance, the tableau
\[
\begin{tabular}{|c|c|c|c|}
\hline
1&3&4&6\\
\hline
2&5&7&9\\
\hline
8&10&11&12\\
\hline
\end{tabular}
\]
corresponds to the lattice path
\begin{align*}
&(3,2,1),(4,2,1),(5,2,1),(5,3,1),(5,3,2),(5,4,2),(4,3,1),\\
&(4,3,2),(5,3,2),(4,2,1),(4,3,1),(4,3,2),(3,2,1).
\end{align*}
The associated divisor has multiplicity three at $v_0$ and one carefully chosen point on each cell $\gamma_i$ for $i \neq 6,9,12$.  The locations of these points can be read off from the tableau and lattice path, as follows.  The columns of the tableau are labeled from zero to three and 10 appears in the column labeled one, so the point in $\gamma_{10}$ is at distance
\[
x_{10} \equiv 3 m_{10} \mod \ell_{10} + m_{10}
\]
from $v_9$, in the counterclockwise direction, with the coefficient of $m_{10}$ given by the formula $p_9(1) + 1 = 3$.
\end{ex}

\begin{rem} \label{rem:Game}
One may think of a divisor of degree $d$ and rank at least $r$ on $\Gamma$ as a winning strategy in the following game, which we call the \emph{Brill-Noether game}.  First, Brill chooses a divisor $D$ of degree $d$ on $\Gamma$, and shows it to Noether.  Then, Noether chooses an effective divisor $E$ of degree $r$ on $\Gamma$ and shows it to Brill.  Finally, Brill performs chip-firing operations on $D-E$, and wins if he reaches an effective divisor.  Otherwise Noether wins.  One can imagine a similar game with divisors on an algebraic curve.

Our main result says that Brill wins the game if and only if $\rho$ is nonnegative, and our proof is a combinatorial classification of Brill's winning strategies.  In terms of this game, the theory of $v_0$-reduced divisors implies that Brill has an optimal strategy in which he chooses an effective $v_0$-reduced divisor with $d_0 \geq r$, and Luo's Theorem implies that Noether has an optimal strategy in which he chooses a divisor supported in $\{v_0, \ldots, v_g \}$.

When the game is played, Brill typically starts with a large pile of chips at $v_0$ and moves this pile to the right using chip-firing moves, as in Example~\ref{ex:Circle}.  The pile grows when it picks up an additional chip that Brill placed at $x_i$ in the $i$th loop, if $x_i$ is carefully chosen.  The pile shrinks when it encounters Noether's antichips in $-E$, or when it crosses over an empty loop and one chip is left behind in the process of moving the pile to the next vertex.  The dynamics of this growing and shrinking pile of chips are encoded in the coordinates of the points in the lattice path $P$.  Roughly speaking, the $j$th coordinate $p_i(j)$ is the size of the pile when it reaches $v_i$, if Noether has distributed $j$ antichips over the vertices $v_0, \ldots, v_i$, assuming optimal play.  For instance, if Noether places $j$ antichips at $v_0$, then Brill is left with $p_0(j) = d_0 - j$ chips at $v_0$.
\end{rem}

The preceding remarks are made precise in the following proposition, which will be used in the proof of Theorem~\ref{thm:Technical}.
As before, let $\gamma_i$ be the $i$th loop minus $v_{i-1}$, which is
an open cell in the decomposition of $\Gamma \smallsetminus v_0$,
described in Example~\ref{ex:Reduced}.  The restriction of a divisor
$D = a_1 w_1 + \cdots + a_s w_s$ to $\gamma_j$ is defined as
\[
D|_{\gamma_j} = \sum_{w_i \in \gamma_j} a_i w_i.
\]

\begin{prop} \label{prop:Induction}
Suppose $p_0, \ldots, p_{n-1}$ are in $\cC$.  Let $E_n$ be an effective divisor of degree $j < r$ with support contained in $\{v_0, \ldots, v_n\}$, and let $D_n$ be the $v_n$-reduced divisor equivalent to $D - E_n$.  Then
\begin{enumerate}
\item \label{item:Coefficient} the coefficient of $v_n$ in $D_n$ is at least
$p_n(j)$; and
\item \label{item:Restriction} for $i > n$, the restriction $D_n|_{\gamma_i}$  is equal to $D|_{\gamma_i}$.
\end{enumerate}
Furthermore, for each $j < r$ there exists an effective divisor $E_n$ of degree $j$
with support in $\{v_0,\ldots,v_n\}$ such that equality
holds in (\ref{item:Coefficient}).
\end{prop}

\begin{proof}
Write $E = r_0 v_0 + \cdots + r_n v_n$.  If $n = 0$, then $D_0$ is $D - r_0 v_0$, and the proposition is clear.  We proceed to prove (\ref{item:Coefficient}) and (\ref{item:Restriction}) by induction on $n$.  Let $D_{n-1}$ be the $v_{n-1}$-reduced divisor equivalent to $D - r_0 v_0 - \cdots - r_{n-1} v_{n-1}$.  By the induction hypothesis, we may assume that the coefficient $k$ of $v_{n-1}$ in $D_{n-1}$ is at least $p_{n-1}(j - r_n)$ and the restriction of $D_{n-1}$ to $\gamma_i$ is equal to $D|_{\gamma_i}$ for $i \geq n$.

Let $D'_{n-1}$ be the divisor equivalent to $D_{n-1}$ obtained by moving the chips at $v_{n-1}$ to $v_n$, as in Example~\ref{ex:Circle}.  Then $D'_{n-1}$ is effective and each cell of $\Gamma \smallsetminus v_n$ contains at most one point of $D_{n-1}'$, so $D_{n-1}'$ is $v_n$-reduced.  It follows that $D_n = D'_{n-1} - r_n v_n$ and the restriction of $D_n$ to $\gamma_i$ is equal to $D|_{\gamma_i}$ for $i > n$, as required.  We now show that the coefficient of $v_n$ in $D_n$ is at least $p_n(j)$ by considering three cases according to the position of the point of $D$, if any, in $\gamma_n$.
\smallskip

\emph{Case 1.} There is no point of $D$ in $\gamma_n$.  In this case,
the $n$th step of $P$ is in the direction $(-1, \ldots, -1)$, so $p_n(j)$
is equal to $p_{n-1}(j) - 1$.  In terms of the Brill-Noether game, the
pile of chips at $v_{n-1}$ shrinks by one as it moves from $v_{n-1}$
to $v_n$, since one chip must be left behind in $\gamma'_n$.  In other
words, the equivalence $D'_{n-1} \sim D_{n-1}$ is given by the first
case in Example~\ref{ex:Circle}, and hence the coefficient of $v_n$
in $D_{n-1}'$ is $k-1$, and that of $D_n$ is $k-1-r_n$. Now $k$ is at
least $p_{n-1}(j-r_n)$, which is at least $p_{n-1}(j) + r_n$, since the
coordinates of $p_{n-1}$ are strictly decreasing integers.  Therefore, the
coefficient of $v_n$ in $D_n$ is greater than or equal to $p_{n-1}(j)-1$,
as required. For later use note that equality holds if $k=p_{n-1}(j-r_n)$
and moreover the entries of $p_i$ (and hence of $p_{i-1}$) at positions
$j-r_n,\ldots,j$ are consecutive integers.

\bigskip

\emph{Case 2.}  The point of $\gamma_n$ in $D$ is at $x_n \equiv
(p_{n-1}(j - r_n)+1) m_n \mod \ell_n + m_n$.  By hypothesis $k$ is at
least $p_{n-1}(j-r_n)$.  If it is equal to $p_{n-1}(j-r_n)$ then the
pile of chips picks up one extra as it moves from $v_{n-1}$ to $v_n$.
In other words, the equivalence $D_{n-1} \sim D'_{n-1}$ is given by the
second case in Example~\ref{ex:Circle}, and hence the coefficient of $v_n$
in $D_{n}$ is $p_{n-1}(j-r_n) + 1 - r_n \geq p_n(j-r_n) -r_n \geq p_n(j)$,
as required. For later use note that equality holds if $p_{n-1}(j-r_n)+1$
does not occur among the entries of $p_{n-1}$ and moreover the entries of $p_n$
at positions $j-r_n,\ldots,j$ are consecutive integers.

On the other hand, if $k$ is greater than $p_{n-1}(j-r_n)$ then the
equivalence $D_{n-1} \sim D'_{n-1}$ is given by the third case in
Example~\ref{ex:Circle}, and the coefficient of $v_n$ in $D_n$
is $k - r_n$, which is again greater than or equal to $p_n(j)$.

\bigskip

\emph{Case 3.}  There is a point of $\gamma_n$ in $D$, but not at $x_n$.
In this case, the pile does not shrink as it moves from $v_{n-1}$ to
$v_n$, and $p_n(j-r_n)$ is equal to $p_{n-1}(j-r_n)$.  So the coefficient
of $v_n$ in $D_n$ is at least $k -r_n \geq p_n(j)$, as required.  It is worth noting that if $k$ is greater than $p_{n-1}(j-r_n)$ and the point of $D$ is at $(k+1)m_n$ then the pile will grow as it moves from $v_{n-1}$ to $v_n$, as in Case 2, but the lattice path still lingers.  However, if $k=p_{n-1}(j-r_n)$ then the pile does not grow, and moreover the entries of $p_n$ at positions $j-r_n,\ldots,j$ are consecutive integers.  

\bigskip

It remains to show that $E_n$ can be chosen so that equality holds
in (\ref{item:Coefficient}). Again, we proceed by induction on
$n$. For $n=0$ we can take $E_n=j v_0$. For the induction step from
$n-1$ to $n$ let $r_n$ be maximal such that the entries of $p_n$
at positions $j-r_n,\ldots,j$ are consecutive integers, and let
$E_{n-1}$ be an effective divisor of degree $j-r_n$ with support in
$\{v_0,\ldots,v_{n-1}\}$ such that the coefficient $k$ of $v_{n-1}$
of the $v_{n-1}$-reduced divisor equivalent to $D-E_{n-1}$ equals
$p_{n-1}(j-r_n)$. To prove that $E_n:=E_{n-1}+r_n v_n$ has the required
property we go through the Cases 1--3 above. In Cases 1 and 3 the
equality $k=p_{n-1}(j-r_n)$ and the fact that the entries of $p_n$
at positions $j-r_n,\ldots,j$ are consecutive suffice to conclude that
the coefficient of $v_n$ in $E_n$ equals $p_{n-1}(j)$. In Case 2 the
only thing that could go wrong is that the pile of $k=p_{n-1}(j-r_n)$
chips at $v_{n-1}$ picks up an additional chip when moved to $v_n$ but
$p$ lingers at the $n$-th step since the entries of $p_{n-1}+e_{j-r_n}$
are not all distinct. Then we have $j-r_n>0$ and the entry of $p_{n-1}$
at position $j-r_n-1$ equals $p_{n-1}(j-r_n)+1$, which is then also
the entry of $p_n$ at position $j-r_n-1$. But then the entries of
$p_n$ at positions $j-r_n-1,\ldots,j$ are consecutive integers,
contradicting the maximality of $r_n$. We conclude that $E_n$ has the
required property.
\end{proof}

We conclude by applying the proposition to prove Theorem~\ref{thm:Technical}.

\begin{proof}[Proof of Theorem~\ref{thm:Technical}]
Suppose the lattice path $P$ lies in the open Weyl chamber $\cC$.  Let $E = r_0 v_0 + \cdots + r_g v_g$ be an effective divisor of degree $r$, and let $n$ be the largest index such that $r_n$ is strictly positive.  Let $E_n = E - v_n$.  By Proposition~\ref{prop:Induction}, the difference $D - E_n$ is equivalent to an effective divisor $D'_n$ in which the coefficient of $D'_n$ is at least $p_n(r-1)$, which is strictly positive since $P$ is in $\cC$.  Therefore $D - E$ is equivalent to the effective divisor $D'_n - v_n$.  Since $E$ is an arbitrary effective divisor of degree $r$ with support in $\{v_0, \ldots, v_g \}$, it follows by Luo's Theorem that $D$ has rank at least $r$.

For the converse, suppose the lattice path $P$ does not lie in $\cC$, and let $n$ be the smallest index such that $p_n$ is not in $\cC$.  By the construction of $P$, all coordinates of $p_i$ are nonnegative for $i \leq n$, and $p_n(r-1) = 0$.  By Proposition~\ref{prop:Induction}, there exists an effective divisor $E_n = r_0 v_0 + \cdots + r_n v_n$ of degree $r-1$ such that the coefficient of $v_n$ in the $v_n$-reduced divisor $D'_n$ equivalent to $D - E_n$ is zero.  Then $E = E_n + v_n$ is an effective divisor of degree $r$, and the $v_n$-reduced divisor equivalent to $D - E$ is $D'_n - v_n$, which is not effective.  Therefore, $D - E$ is not equivalent to any effective divisor, and hence $D$ has rank less than $r$, as required.
\end{proof}

\bibliography{math}
\bibliographystyle{amsalpha}

\end{document}